\documentclass[orivec,runningheads,envcountsame]{llncs}

\usepackage{graphicx}
\usepackage{amsmath}
\usepackage{amssymb}

 \newcommand{\xcs} {{\cal S}}

\newcommand{\inverse}[1]{ ( #1 )^{-1}}

\begin{document}             % End of preamble and beginning of text.

\setlength{\parindent}{0mm}
 \setlength{\parskip}{4pt}

\title{Subrecursive Approximations of Irrational Numbers by Variable Base Sums}

\author{Ivan Georgiev\inst{1}\\Lars Kristiansen\inst{2,3}\\Frank Stephan\inst{4}}

\authorrunning{I. Georgiev, L. Kristiansen and F. Stephan}

 \institute{
Department of Mathematics and Physics, Faculty of Natural Sciences, \\ University "Prof. d-r Asen Zlatarov", Burgas 8010, Bulgaria\setcounter{footnote}{0}\renewcommand{\thefootnote}{\fnsymbol{footnote}}\footnote{I. Georgiev has been supported by the Bulgarian National Science Fund through the project "Models of computability", DN-02-16/19.12.2016.} \\
\email{ivandg@yahoo.com}
\and
Department of Mathematics, University of Oslo, Norway
\and
    Department of Informatics, University of Oslo, Norway \\
		\email{larsk@math.uio.no}
\and
Department of Mathematics and School of Computing, \\ 
National University of Singapore, Singapore 119076, Republic of Singapore\footnote{F. Stephan has been supported in part by the Singapore Ministry
of Education Academic Research Fund grant MOE2016-T2-1-019 / R146-000-234-112.} \\
\email{fstephan@comp.nus.edu.sg}
}

 \maketitle                   % Produces the title.

\newcommand{\xd}{\texttt{D}}
\newcommand{\integer}{\ensuremath{\mathbb Z}}
\newcommand{\rational}{\ensuremath{\mathbb Q}}
\newcommand{\nat}{\ensuremath{\mathbb N}}
\newcommand{\real}{\ensuremath{\mathbb R}}
\newcommand{\kleeneT}{{\mathcal T}}

\newcommand{\kleeneU}{{\mathcal U}}

\section{Introduction and Basic Definitions}

There are numerous ways to represent real numbers. We may use, e.g., Cauchy sequences, Dedekind cuts, numerical base-10 expansions,
numerical base-2 expansions and continued fractions. If we work with full Turing computability, all these representations
yield the same class of real numbers. If we work with some restricted notion of computability, e.g., polynomial time
computability or primitive recursive computability, they do not. This phenomenon has been investigated over
the last  seven decades by Specker \cite{specker}, Mostowski \cite{most}, Lehman \cite{lehman}, Ko \cite{koa,kob},
Labhalla \& Lombardi \cite{ll},
Georgiev \cite{geo}, Kristiansen  \cite{ciejour,tociejour} and quite a few more.

Irrational numbers can be represented by infinite sums.
Sum approximations from below and above were introduced in Kristiansen \cite{ciejour} and studied further in Kristiansen \cite{tociejour}.
Every irrational number $\alpha$ between 0 and 1 can be uniquely written as an infinite sum of the form
$$
\alpha \;\; = \;\; 0 \; + \; \frac{\xd_1}{b^{k_1}} \; + \; \frac{\xd_2}{b^{k_2}} \; + \; \frac{\xd_3}{b^{k_3}} \; + \; \ldots
$$
where 
\begin{itemize}
\item $b\in\nat\setminus\{0,1\}$ and $\xd_i\in \{1,\ldots, b-1\}$ (note that $\xd_i\neq 0$ for all $i$)
\item $k_i  \in\nat\setminus\{0\}$    and  $k_i<k_{i+1}$.
\end{itemize}
Let $\hat{A}^\alpha_b(i)=  \xd_ib^{-k_i}$ for $i>0$ (and let  $\hat{A}^\alpha_b(0)=0$).
The rational number $\sum_{i=1}^n \hat{A}^\alpha_b(i)$ is an approximation of $\alpha$ that lies below $\alpha$,
and we will say that the function  $\hat{A}^\alpha_b$ is the {\em base-$b$ sum approximation from below of $\alpha$}.
The  {\em base-$b$ sum approximation from above of $\alpha$} is a symmetric function  $\check{A}^\alpha_b$
such that $1- \sum_{i=1}^n \check{A}^\alpha_b(i)$ is an approximation of $\alpha$ that lies above $\alpha$ (and we
have $\sum_{i=1}^\infty \hat{A}^\alpha_b(i) + \sum_{i=1}^\infty \check{A}^\alpha_b(i) = 1$).
Let $\xcs$ be any  class of subrecursive functions which is closed under primitive recursive operations. Furthermore,
let $\xcs_{b\uparrow}$ denote the set of irrational numbers that have a 
base-$b$ sum approximation  from below in   $\xcs$, and
let $\xcs_{b\downarrow}$ denote the set of irrational numbers that have a 
base-$b$ sum approximation  from above in $\xcs$. It is proved in   \cite{tociejour}
 that  $\xcs_{b\uparrow}$ and  $\xcs_{b\downarrow}$ are incomparable classes, that is,
 $\xcs_{b\uparrow}  \not\subseteq  \xcs_{b\downarrow}$ and  $\xcs_{b\downarrow}  \not\subseteq  \xcs_{b\uparrow}$.
Another interesting result proved in \cite{tociejour} is that 
$\xcs_{a\downarrow} \subseteq  \xcs_{b\downarrow}$  iff 
$\xcs_{a\uparrow} \subseteq  \xcs_{b\uparrow}$ iff 
every prime factor of $b$ is  a prime factor of $a$. 

In this paper we prove some results on {\em general sum approximations}.
The {\em general sum approximation from below} of $\alpha$ is the function $\hat{G}^{\alpha}:\nat\times\nat\rightarrow \rational$ defined
by $\hat{G}^{\alpha}(b,n)= \hat{A}^{\alpha}_b(n)$; let $\hat{G}^{\alpha}(b,n)=0$ if $b<2$.
The {\em general sum approximation from above} of $\alpha$ is the function $\check{G}^{\alpha}:\nat\times\nat\rightarrow \rational$ defined
by $\check{G}^{\alpha}(b,n)= \check{A}^{\alpha}_b(n)$; let $\check{G}^{\alpha}(b,n)=0$ if $b<2$.
Let $\xcs$ be any class of subrecursive functions which is closed under primitive recursive operations. 
Furthermore, let $\xcs_{g\uparrow}$ denote the set of irrational numbers that have a 
general sum approximation from below  in $\xcs$, and let $\xcs_{g\downarrow}$ denote the set of irrational numbers that have a 
general sum approximation from above  in $\xcs$.

It was proved in  \cite{ciejour} that 
$\xcs_{g\uparrow}\cap \xcs_{g\downarrow}$ contains exactly the irrational numbers that have a continued fraction in the class $\xcs$.
In this paper we prove that $\xcs_{g\uparrow}\neq \xcs_{g\downarrow}$. Moreover, we prove that
$$ \xcs_{g\downarrow} \;\; \neq \;\; \bigcap_{b=2}^\infty \xcs_{b\downarrow}\;\;\;\mbox{ and } 
\;\;\;  \xcs_{g\uparrow} \;\; \neq \;\; \bigcap_{b=2}^{\infty} \xcs_{b\uparrow} .$$
Some might find it interesting (at least the authors do) that we  manage to complete all our proof without resorting to
the standard computability-theoretic machinery involving enumerations, universal functions, diagonalizations, and so on.
We prove our results by providing natural  irrationals numbers (the numbers are natural in the sense that they
have neat and transparent definitions).

\section{Preliminaries}

We will restrict our attention to real numbers between 0 and 1.

A {\em base} is a natural number strictly greater than 1, and a {\em base-$b$ digit} is a natural number in the
set $\{0,1,\ldots, b-1\}$.

Let $b$ be a base, and let $\xd_1,\ldots, \xd_n$ be base-$b$ digits.  We
will use $(0.\xd_1\xd_2\ldots \xd_n)_b$ to denote the rational number 
$\sum_{i=1}^n \xd_i b^{-i}$.

Let  $\xd_1,\xd_2,\ldots$ be an infinite sequence of base-$b$ digits. 
We say that $(0.\xd_1\xd_2\ldots)_b$ is {\em the base-$b$ expansion of the real number $\alpha$} if for all $n \geq 1$ we have
$$
(0.\xd_1\xd_2\ldots \xd_n)_b \;\; \leq \;\; \alpha \;\; < \;\;  (0.\xd_1\xd_2\ldots \xd_n)_b + b^{-n}.
$$
Every real number $\alpha$ has a unique base-$b$ expansion (note the strict inequality).

When $\alpha =(0.\xd_1\xd_2\ldots \xd_n)_b$ for some $n$ with $\xd_n \neq 0$, 
we say that $\alpha$ has a {\em finite base-$b$ expansion} of length $n$. 
Otherwise, we say that $\alpha$ has an {\em infinite base-$b$ expansion}, and this infinite base-$b$ expansion
 is periodic iff $\alpha$ is rational. 
More concretely, if $\alpha = cd^{-1}$ for 
non-zero relatively prime $c,d\in\nat$, then the base-$b$ expansion of $\alpha$ is 
of the form $\xd_1\ldots \xd_{s}(\xd_{s+1}\ldots \xd_t)^\omega$ which we use as shorthand for the infinite sequence 
$\xd_1\ldots \xd_{s}\xd_{s+1}\ldots \xd_t\xd_{s+1}\ldots \xd_t\xd_{s+1}\ldots \xd_t\ldots$.
Let $d=d_1d_2$, where $d_1$ is relatively prime to $b$ and largest possible. Then $s$ is the least natural number, such that $d_2$ divides $b^s$ and the length of the period $t-s$ is the multiplicative order of $b$ modulo $d_1$. It  follows straightforwardly that $t < d$. 
Of course, $\alpha$ has a finite base-$b$ expansion iff $d_1 = 1$, that is, 
iff every prime factor of $d$ is a prime factor of $b$.

We assume the reader is familiar with subrecursion theory and subrecursive functions.
An introduction  to the subject can be found in \cite{peter} or \cite{rose}.

A function $\phi$ is {\em elementary in} a function $\psi$,
written $\phi \leq_{E} \psi$, if
$\phi$ can be generated from the initial functions
$\psi$, $2^x$, $\max$, $0$, $S$ (successor),  $I^n_i$  (projections)
by  composition and bounded primitive recursion. A function $\phi$ is {\em elementary} if
$\phi \leq_{E} 0$.
A function $\phi$ is {\em primitive recursive in} a function $\psi$,
written $\phi \leq_{PR} \psi$, if
$\phi$ can be generated from the initial functions
by  composition and (unbounded) primitive recursion. A function $\phi$ is {\em primitive recursive} if
$\phi \leq_{PR} 0$.

Subrecursive functions in general, and  elementary functions in particular, are
formally functions over natural numbers ($\nat$). We assume some coding of integers ($\integer$) 
and rational numbers ($\rational$)  into the natural numbers. We consider such a coding to
be trivial. Therefore we allow for subrecursive functions from rational numbers  into natural numbers,
from pairs of rational numbers into rational numbers, etc., with no further comment.
Uniform systems for coding  finite sequences of natural numbers
are available inside the class of elementary functions.
Hence, for any reasonable coding,
 basic operations on rational numbers -- like e.g.\ addition, subtraction and multiplication --
will obviously be elementary. It is also obvious that there is an elementary function $\psi(q,i,b)$ that yields the $i^{\mbox{\scriptsize th}}$
digit in the base-$b$ expansion of the rational number $q$.

 A function $f:\nat\rightarrow \nat$ is {\em honest}
 if it is  monotonically increasing ($f(x)\leq f(x+1)$), dominates $2^x$  ($f(x)\geq 2^x$)
and has elementary graph (the relation $f(x)=y$ is elementary). 

A class of functions $\xcs$ is
 {\em subrecursive class} if $\xcs$ is an efficiently enumerable class of computable total functions.
For any subrecursive class $\xcs$ there exists an honest function $f$ such that $f\not\in \xcs$
(see Section 8 of \cite{ciejour} for more details).

More on elementary functions, primitive recursive functions and honest functions can be found in Section 
2 of \cite{ciejour} and in  \cite{apal}.

\section{Irrational Numbers with Interesting Properties}

\begin{definition}
Let $P_i$ denote the $i^{\mbox{\scriptsize th}}$ prime ($P_0=2, P_1=3,\ldots$). 
We define the auxiliary  function $g$ by 
$$ g(0) \; = \; 1 \;\;\; \mbox{ and } \;\;\; g(j+1) \; = \; P_j^{2(j+2)(g(j)+1)^3}\; . $$

For any honest function $f$ and any $n\in\nat$, we define the rational number $\alpha^f_n$  and the number $\alpha^f$ by
$$ \alpha^f_n \; = \; \sum_{i=0}^n P_i^{-h(i)}\;\;\;\; \mbox{ and }\;\;\;\;
 \alpha^f \; = \; \lim_{n\rightarrow \infty}  \alpha^f_n$$
where $h(i) = g(f(i)+i)$ (for any $i\in\nat$).
\end{definition}

It is easy to see that $g$ and $h$ are strictly increasing honest functions. Only the fact that $g$ has elementary graph requires some explanation:
for all $x,y\in\nat$ the equality $g(x) = y$ holds iff
$$ \exists s\; ((s)_0 = 1\;\;\;\&\;\;\; \forall j \leq x \;((s)_{j+1} \;=\;  P_j^{2(j+2)((s)_j+1)^3})\;\;\;\&\;\;\;(s)_x \;=\; y). $$
We have that $s$ is the code of a sequence of length $x+1$ and its $j$-th element $(s)_j$ is equal to $g(j)$ for all $j \leq x$. Since $g$ is increasing and the coding of sequences can be chosen to be monotonic, the existential quantifier on $s$ can be bounded by the code of the sequence $y, \ldots, y$ ($x+1$ times), which is elementary in $x,y$.

The function $h$ possesses the following growth property, which will need:
\begin{equation}\label{proph}
P_{n}^{2(n+2)(h(n)+1)^3} \;\; < \;\; h(n+1)
\end{equation}
for any $n\in\nat$. The function $f$ itself might not possess it and this explains why we introduce the function $g$ in the definition of $\alpha^f$.

When $f$ is a fixed honest function, we abbreviate $\alpha^f_j$ and $\alpha^f$ to $\alpha_j$ and $\alpha$, respectively.

\begin{lemma}\label{irrat}
The number $\alpha^f$ is irrational.
\end{lemma}
\begin{proof} Assume that $\alpha^f = \frac{p}{q}$ for some natural numbers $p,q$ with $q$ non-zero. Let $q_i$ be the exponent of $P_i$ in the prime factorization of $q$.
We choose $n$, such that $P_n$ is greater than all prime factors of $q$ and $h(n)$ is greater than $q_0 + q_1 + \ldots + q_n$.
It is easy to see that $\alpha^f - \alpha_n^f$ is a non-zero rational number with denominator $P_0^{h(0)'}P_1^{h(1)'}\ldots P_n^{h(n)'}$,
where $h(i)' = h(i) + q_i$. Therefore,
$$ \frac{1}{P_0^{h(0)'}P_1^{h(1)'}\ldots P_n^{h(n)'}} \;\;\leq\;\; \alpha^f \;-\; \alpha_n^f. $$
But we also have
$$ \alpha^f \;-\; \alpha_n^f \;\;=\;\; P_{n+1}^{-h(n+1)} \;+\; P_{n+2}^{-h(n+2)} \;+\; \ldots $$
and since $h$ is strictly increasing,
$$ P_{n+1}^{-h(n+1)} \;+\; P_{n+2}^{-h(n+2)} \;+\; \ldots \;\;\leq\;\; P_{n+1}^{-h(n+1)} \;+\; P_{n+1}^{-h(n+1)-1} \;+\; \ldots \;\;\leq\;\; P_{n+1}^{-h(n+1)+1}. $$
We obtained the inequalities
$$ \frac{1}{P_0^{h(0)'}P_1^{h(1)'}\ldots P_n^{h(n)'}} \;\;\leq\;\; \alpha^f \;-\; \alpha_n^f \;\;\leq\;\ P_{n+1}^{-h(n+1)+1}, $$
which imply in turn
$$ P_{n+1}^{h(n+1)-1} \;\;\leq\;\; P_0^{h(0)'}P_1^{h(1)'}\ldots P_n^{h(n)'}, $$
$$ P_{n+1}^{h(n+1)-1} \;\;\leq\;\; P_{n+1}^{h(0)'+h(1)'+\ldots+h(n)'}, $$
$$ h(n+1) \;\;\leq\;\; (n+1)h(n) + q_0 + \ldots + q_n + 1, $$
$$ h(n+1) \;\;\leq\;\; (n+2)h(n) + 1. $$
But the last inequality is easily seen to be false by (\ref{proph}). Contradiction. $\qed$
\end{proof}

In fact, it is not hard to show using (\ref{proph}) that $\alpha^f$ is a Liouville number, therefore it is even transcendental.

\begin{lemma} \label{abrahamlincoln}
For any  $j\in\nat$ and any base $b$, we have
\begin{itemize}
 \item[(i)] if  $P_i$ divides $b$ for all $i \leq j$, then $\alpha_j$ has a finite base-$b$ expansion of length $h(j)$
 \item[(ii)] if  $P_i$ does not divide $b$ for some $i \leq j$, then $\alpha_j$ has an infinite (periodic) base-$b$ expansion.
\end{itemize}
\end{lemma}
\begin{proof}
For any $j\in\nat$ and base $b$ we have that
$$ \alpha_j = \frac{m}{\prod_{i=0}^j P_i^{h(i)}}, $$
where $m$ is the sum of $j+1$ summands. Each of these summands is divisible by all $P_i$ for $i\leq j$ with the exception of exactly one of them. From this it follows that $m$ is relatively prime to all $P_i$ for $i\leq j$, so the above fraction is in lowest terms. The rest follows easily from the preliminaries on base-$b$ expansions. $\qed$
\end{proof}
%\begin{proof}
%Observe that the  base-$b$ expansion $\alpha_j$ is finite iff there exist $k,m\in\nat$ such that $\alpha_j= mb^{-k}$.
%It follows  from the definition above that 
%$\alpha_j = t(P_0P_1\ldots P_j)^{-h(j)}$ for some $t\in\nat$ where $P_j$ does not divide $t$.
%Assume for the sake of a contradiction that 
%there exist $k,m\in\nat$ such that $t(P_0P_1\ldots P_j)^{-h(j)}= mb^{-k}$.
%Then we have $m(P_0P_1\ldots P_j)^{h(j)} = tb^{k}$. This contradicts that $tb^{k}$ has a unique prime factorization as
%$P_j$ does not divide $tb^{k}$.
%\qed
%\end{proof}

\begin{lemma} \label{barackobama}
Let 
\begin{itemize} 
\item $b$ be any base, and let $j\in\nat$ be such that $P_j> b$
\item $(0.\xd_1\xd_2\ldots )_b$ be the base-$b$ expansion of  $\alpha_j$
\item $(0.\dot{\xd}_1\dot{\xd}_2\ldots )_b$ be the base-$b$ expansion of $\alpha_{j+1}$
\item  $M = M(j) = P_j^{(j+1)h(j)}$ and $M' = M'(j) = h(j+1)$.
\end{itemize} 
Then
\begin{itemize}
 \item[(i)] there are no more than $M$ consecutive zeros in the base-$b$ expansion of $\alpha_j$, that is,
 for any $k\in\nat\setminus\{0\}$ there exists $i\in\nat$
such that $$k \leq i < k + M \;\;\mbox{ and }\;\; \xd_i \neq 0$$
 \item[(ii)] the first $M'-M$ digits of the base-$b$ expansions of $\alpha_j$ and $\alpha_{j+1}$ coincide, that is
$$ i \leq M' - M \;\;\Rightarrow\;\; \xd_i = \dot{\xd}_i $$
and moreover, these digits also coincide with the corresponding digits of the base-$b$ expansion of $\alpha$.
\end{itemize}
\end{lemma}

\begin{proof}
By Lemma \ref{abrahamlincoln} (ii), $\alpha_j$ has an infinite 
periodic base-$b$ expansion of the form $0.\xd_1\ldots \xd_s(\xd_{s+1}\ldots \xd_t)^\omega$ with $s < t$. 
Using the preliminaries on base-$b$ expansions we obtain
\begin{align}
 t - s \;\;\leq \;\; t \;\; < \;\; \prod_{i=0}^j P_i^{h(i)} \;\; \leq \;\; P_j^{(j+1)h(j)} \;\; = \;\; M \; . \label{tobarac}
\end{align}
Thus (i) holds since every $M$ consecutive digits of $\alpha_j$ contain all the digits $\xd_{s+1},\ldots, \xd_t$ of at least one period. 

We turn to the proof of (ii). We have
\begin{align}
\alpha_j \;\; < \;\; \alpha_{j+1} \;\; = \;\; \alpha_j \, + \, P_{j+1}^{-h(j+1)} \;\; \leq \;\; \alpha_j \, + \, b^{-M' } \label{nybarac}
\end{align}
since $b^{M'} < P_j^{M'} = P_j^{h(j+1)} < P_{j+1}^{h(j+1)}$.
At least one digit in the period $\xd_{s+1}\ldots \xd_t$ is different from $b-1$, and the length of the period is $t-s$.
Thus, it follows from (\ref{nybarac}) that
\begin{align}
  \mbox{$\xd_i = \dot{\xd}_i$ for any $i \leq M' - (t-s)$. }\label{nynybarac}
\end{align}
It follows from  (\ref{tobarac}) and (\ref{nynybarac}) that
 the first $M'-M$ digits of the base-$b$ expansions of $\alpha_j$ and $\alpha_{j+1}$ coincide.
 Moreover, since $M'(j)$ is strictly increasing in $j$, we have
$$ \alpha_j \;\; < \;\; \alpha_{j+k} \;\; \leq \;\; \alpha_j \;\;+\;\; \sum_{i<k} b^{-M'(j+i)} \;\; \leq \;\; \alpha_j \;\; + \;\;b^{-M'(j)+1} $$
for any $k \geq 1$. Letting $k\to\infty$ we obtain as above that the first $M'-M$ digits of $\alpha_j$ and $\alpha$ coincide.
\qed
\end{proof}

In the proof of the next theorem the following inequality will be needed:
\begin{equation}\label{sqineq}
 M(j)^2 + M(j) + 1 < M'(j) 
\end{equation}
for all $j\in\nat$. To prove it apply (\ref{proph}) together with
$$ P_j^{2(j+1)h(j)} + P_j^{(j+1)h(j)} + 1 < 3P_j^{2(j+1)h(j)} < P_j^{2(j+1)h(j)+2} $$
for all $j\in\nat$. 

\begin{theorem} \label{sebastian}
Let $f$ be any  honest function, and let $b$ be any base.
The function $\hat{A}^{\alpha^f}_b$ is elementary.
\end{theorem}

\begin{proof}
Fix the least $m$  such that  $P_m> b$. We will use the functions $M$ and $M'$ from Lemma \ref{barackobama}. We will argue that we can compute the rational number $\hat{A}^{\alpha}_b(n)$ elementarily in $n$ when $n\geq M(m)$. Note that $M(m)$ is a fixed number (it does not depend on $n$). Thus,  we can compute $\hat{A}^{\alpha}_b(n)$ by a trivial algorithm  when $n < M(m)$ (use a huge table).

Assume $n\geq M(m)$. We will now give an algorithm for computing $\hat{A}^{\alpha}_b(n)$ elementarily in $n$.

\textit{Step 1 of the algorithm:} Compute (the unique) $j$ such that 
\begin{align}
M(j) \; \leq \; n  \; < \; M(j+1) \label{ensebastian}
\end{align}
{\em (end of Step 1)}.

Step 1 is a computation elementary in $n$ since $M$ has elementary graph.  
So is Step 2 as $M'$ also has elementary graph.

\textit{Step 2 of the algorithm:} Check if the relation
\begin{align}
n^2 + 1 \;\; < \;\; M'(j) \; - \; M(j) \label{tosebastian}
\end{align}
holds. If it holds, carry out step 3A below, otherwise, carry out step 3B
{\em (end of Step 2)}.

\textit{Step 3A of the algorithm:}
Compute $\alpha_j$. Then 
compute base-$b$ digits $\xd_1,\ldots,\xd_{n^2+1}$ such that
$$
(0.\xd_1\xd_2\ldots\xd_{n^2+1})_b \;\; \leq \;\; \alpha_j 
< \;\; (0.\xd_1\xd_2\ldots\xd_{n^2 + 1})_b\, + \,  b^{-(n^2+1)}\; .
$$
Find $k$ such that $\xd_k$ is the $n^{\mbox{\scriptsize th}}$ digit different from 0 in the sequence $\xd_1,\ldots,\xd_{n^2+1}$.
Give the output $\xd_k b^{-k}$ {\em (end of Step 3A)}.

Recall that $\alpha_j  =  \sum_{i=0}^j P_i^{-h(i)}$. We can  compute 
$\alpha_j$ elementarily in $n$ since $h(0),h(1),\ldots,h(j) < M(j) \leq n$ and $h$ is honest. Thus, we can also compute the  
base-$b$ digits $\xd_1,\xd_2,\ldots,\xd_{n^2+1}$  elementarily in $n$. In order to prove that our algorithm is correct, we must verify that
\begin{itemize}
\item[(A)] at least $n$ of the digits $\xd_1,\xd_2,\ldots,\xd_{n^2+1}$ are different from 0, and
\item[(B)]   $\xd_1,\xd_2,\ldots,\xd_{n^2+1}$ coincide with the first $n^2 + 1$ digits of $\alpha$.
\end{itemize}

By Lemma \ref{barackobama} (i) the sequence
$
\xd_{kM(j)+1},\xd_{kM(j)+2 },\ldots,\xd_{(k+1)M(j)}
$
 contains at least one non-zero digit (for any  $k\in\nat$). Thus, (A) holds since $n\geq M(j)$. Using (\ref{tosebastian}) and Lemma \ref{barackobama} (ii) we see that (B) also holds. This proves that the output $\xd_k b^{-k} = \hat{A}^\alpha_b(n)$.

\textit{Step 3B of the algorithm:}
Compute $\alpha_{j+1}$ and $M(j+1)$. Then proceed as in step 3A with $\alpha_{j+1}$ in place of $\alpha_j$ and $nM(j+1)$ in place of $n^2$ {\em (end of Step 3B)}.

Step 3B is only executed when  $M'(j) - M(j) \leq n^2+1$. 
Thus, we have $M'(j) = h(j+1) \leq n^2 + n + 1$. 
This entails that we can compute $h(j+1)$ -- and also $\alpha_{j+1}$ and $M(j+1)$ -- elementarily in $n$. 
The inequality (\ref{sqineq}) gives that
$$ M(j+1)^2 \; + \; M(j+1) \; + \; 1 \;\;<\;\; M'(j+1) $$
which  together with (\ref{ensebastian}) imply
$$ nM(j+1) \;+\; 1 \;\;<\;\; M'(j+1) \;-\; M(j+1) \; . $$
As in step 3A, there will be at least  $n$ non-zero digits among the
 first $nM(j+1)$ digits of $\alpha_{j+1}$. Moreover, the  first $nM(j+1)$ digits of $\alpha_{j+1}$ coincide
 with the corresponding digits of $\alpha$.
\qed
\end{proof}

\begin{theorem} \label{finbeck}
Let $f$ be any  honest function. We have
$ f \leq_{PR}  \hat{G}^{\alpha^f}$ ($f$ is primitive recursive in $ \hat{G}^{\alpha^f}$).
\end{theorem}

\begin{proof}
Fix $n\in\nat$, and let $b$ be the base $b= \prod_{i=0}^n P_i$. 
By Lemma (\ref{abrahamlincoln}) (i), $\alpha_n$ has a finite base-$b$ expansion of length $h(n)$. By the definition of 
$\alpha$, we have
$$ \alpha  \;\; = \;\; \alpha_n \; + \; P_{n+1}^{-h(n+1)} \; + \;  P_{n+2}^{-h(n+2)} \; + \; \ldots \; .$$
It follows that for any $j > h(n)$
$$ \hat{G}^\alpha(b,j) \;\;\leq\;\; P_{n+1}^{-h(n+1)} \;+\; P_{n+2}^{-h(n+2)} + \ldots, $$
which easily implies  $ \hat{G}^\alpha(b,j) \leq P_{n+1}^{-h(n+1)+1} $ (use that $h$ is strictly increasing). Hence we also have 
 $\inverse{ \hat{G}^\alpha(b, j) } \geq  P_{n+1}^{h(n+1) - 1} > h(n+1)-1$ for any $j>h(n)$.

The considerations above show that we can compute $h(n+1)$ by the following algorithm:
\begin{itemize}
\item assume that $h(n)$ is computed;
\item compute $b= \prod_{i=0}^n P_i$;
\item search for  $y$ such that $y<\inverse{\hat{G}^\alpha(b, h(n)+1)}+1$ and $h(n+1)=y$;
\item give the output $y$.
\end{itemize}
This  algorithm  is primitive recursive in $\hat{G}^\alpha$: The computation of $b$  is an elementary computation.
The relation $h(x)=y$ is elementary, and thus the search for $y$ is elementary in $h(n)$ and $\hat{G}^\alpha$. 
This proves that $h$ is primitive recursive in $\hat{G}^\alpha$. 
But then  $f$ will also be  primitive recursive in $\hat{G}^\alpha$ as the graph of $f$ is elementary and 
$f(n) \leq h(n)$ (for any $n\in\nat$). This proves that $f \leq_{PR} \hat{G}^\alpha$.
\qed
\end{proof}

\begin{theorem} \label{evalouise}
Let $f$ be any  honest function.  
There exists an elementary function 
 $\check{T}:\rational \rightarrow \rational$ such that
(i) $\check{T}(q) =0$  if  $ q< \alpha^f$ and (ii)
$ q>  \check{T}(q) > \alpha^f$   if  $  q> \alpha^f$.
\end{theorem}

\begin{proof}
In addition to the sequence  $\alpha_j$ we need the sequence $\beta_j$ given by 
$$ \beta_0 \;\;=\;\; P_0^{-h(0)+1} \;\;=\;\; 2^{-h(0)+1} \;\;\;\mbox{ and }\;\;\; \beta_{j+1} \;\;=\;\; \alpha_j \;+\; P_{j+1}^{-h(j+1)+1}\; . $$
Observe that we have  $\alpha < \beta_j$ for all $j\in\nat$, since 
$$ \alpha \;-\; \alpha_j \;\;=\;\; P_{j+1}^{-h(j+1)} \;+\; P_{j+2}^{-h(j+2)} \;+\; \ldots \;\;\leq\;\; P_{j+1}^{-h(j+1)+1} $$
for any $j\in\nat$.

Now we will explain an algorithm that computes a function $\check{T}$ with the properties (i) and (ii).

\textit{Step 1 of the algorithm:} The input is the rational number $q$. We can w.l.o.g.\ assume that $0<q<1$.
Pick any $m',n\in\nat$ such that $q=m'n^{-1}$ and  $n \geq h(0)$.
Find $m\in\nat$ such that $q=m(P_0P_1\ldots P_n)^{-n}$, 
and compute the base $b$ such that $b=\prod_{i=0}^n P_i$ {\em (end of Step 1)}.

The rational number $q$ has a finite base-$b$ expansion of length $s$ where $s\leq n$. Moreover,
the rational numbers $\alpha_0, \alpha_1, \ldots, \alpha_n$ and $\beta_0, \beta_1, \ldots, \beta_n$ also have finite base-$b$ expansions.

\textit{Step 2 of the algorithm:} Compute (the unique) natural number $j < n$ such that
$$ h(j) \;\; \leq \;\; n \;\; < \;\; h(j+1) \; . $$
Furthermore, compute $\alpha_0,\alpha_1,\ldots,\alpha_j$ and $\beta_0,\beta_1,\ldots,\beta_j$ (\textit{end of Step 2}).

All the numbers $h(0),h(1),\ldots h(j)$ are less than or equal to  $n$, and $h$ has elementary 
graph. This entails that Step 2 is elementary in $n$ (and thus also elementary in $q$).

\textit{Step 3 of the algorithm:} If   $q \leq \alpha_k$ for some $k \leq j$, give the output $0$ and terminate.
If  $\beta_k < q$ for some $k \leq j$,  give the output $\beta_k$ and terminate (\textit{end of Step 3}).

Step 3 obviously gives the correct output. It is also obvious that the step is elementary in $q$.

If the algorithm has not yet terminated, we have $\alpha_j < q \leq \beta_j$. Now
$$ q \;\leq\; \beta_{j+1} \;\;\;\Leftrightarrow\;\;\; q \;-\; \alpha_j \;\leq\; P_{j+1}^{-h(j+1)+1} 
\;\;\;\Leftrightarrow\;\;\; P_{j+1}^{h(j+1)} \;\leq\; (q-\alpha_j)^{-1}P_{j+1}. $$
We have determined $\alpha_j$, and $h$ is an honest function. This makes it possible to check elementarily
if $q\leq\beta_{j+1}$: Search for $y< (q - \alpha_{j})^{-1}P_{j+1}$ such that $h(j+1)=y$. If no such $y$ exists, we have 
$q >\beta_{j+1}$. If such an $y$ exists, we have $q\leq\beta_{j+1}$ iff  $P_{j+1}^{y} \leq (q - \alpha_{j})^{-1}P_{j+1}$.

\textit{Step 4 of the algorithm:} Search for $y < (q-\alpha_j)^{-1}P_{j+1}$ such that $y = h(j+1)$. 
If the search  is successful and $P_{j+1}^y \leq (q-\alpha_j)^{-1}P_{j+1}$,  go to Step 5, otherwise 
go to Step 6B (\textit{end of Step 4}).

Clearly, Step 4 is elementary in $q$. If $q \leq \beta_{j+1}$, the next step is Step 5 (and we have computed $y = h(j+1)$). 
If $\beta_{j+1} < q$, the next step is Step 6B.

\textit{Step 5 of the algorithm:} Compute $\alpha_{j+1}$ and $\beta_{j+1}$. If $q \leq \alpha_{j+1}$,
give the output 0 and terminate.
If $\alpha_{j+1} < q$,  search for $z < (q-\alpha_{j+1})^{-1}P_{j+2}$ such that $z = h(j+2)$. 
If the search  is successful and $P_{j+2}^z \leq (q-\alpha_{j+1})^{-1}P_{j+2}$, give the output 0 and terminate, 
otherwise,  go to Step 6A (\textit{end of Step 5}).

Step 5 is elementary in $q$ since we have computed $h(j+1)$ in Step 4. 
If the algorithm terminates because $q \leq \alpha_{j+1}$, we obviously have $q < \alpha$ and the output is correct. 
If   $q > \alpha_{j+1}$,  the algorithm will not proceed to Step 6A 
iff $q \leq \beta_{j+2}$. So assume that $\alpha_{j+1} < q \leq \beta_{j+2}$.

It is a well-known fact that for any natural number $x \geq 2$ there is a prime number between $x$ and $2x$. It follows by induction that $P_y \leq 2^{y+1}$ for all $y\in\mathbb{N}$ and therefore $b = P_0 P_1 \ldots P_n \leq 2^{(n+1)^2}$.

Applying $n < h(j+1)$ together with inequality (\ref{proph}) we obtain
$$ b^{h(j+1)+1} \leq (2^{(n+1)^2})^{h(j+1)+1} < 2^{(h(j+1)+1)^3} < P_{j+2}^{h(j+2)-1} $$
and thus
 $$ \frac{1}{P_{j+2}^{h(j+2)-1}} < \frac{1}{b^{h(j+1)+1}}. $$
This entails
 $$ \alpha_{j+1} < \alpha < \beta_{j+2} < \alpha_{j+1} + \frac{1}{b^{h(j+1)+1}}, $$
but $\alpha_{j+1}$ has length $h(j+1)$, so the first $h(j+1)$ digits of $\alpha_{j+1}, \alpha, \beta_{j+2}$ coincide.
Moreover, $h(j+1) > n \geq s$ (recall that $s$ is the length of the base-$b$ expansion of $q$). Thus, 
we have $q < \alpha$, and the algorithm gives the correct output, namely 0. 
If the algorithm proceeds with Step 6A, we have $\beta_{j+2} < q$. 

\textit{Step 6A of the algorithm:} Compute the least $t$ such that $b^t > (q-\alpha_{j+1})^{-1}$. 
Search for $u < (q - b^{-t} - \alpha_{j+1})^{-1}P_{j+2}$ such that $u = h(j+2)$. 
If the search  is successful and $P_{j+2}^u < (q - b^{-t} - \alpha_{j+1})^{-1}P_{j+2}$, give the output $\beta_{j+2}$
and terminate, otherwise, give the output $q - b^{-t}$ and terminate (\textit{end of Step 6A}).

It is easy to see that Step 6A is elementary in $q$: we can compute $t$ elementarily in $q$, $b$ and $\alpha_{j+1}$ (and we have
already computed $b$ and $\alpha_{j+1}$ elementarily in $q$). When the execution of the step starts, we have $\beta_{j+2} < q$
(thus, $\beta_{j+2}$ will be a correct output, but we do not yet know if we will be able to compute $\beta_{j+2}$). 
If the search for $u$ is successful, we have $u=h(j+2)$. Then, we can compute
$\beta_{j+2}$ elementarily in $u$, and give $\beta_{j+2}$ as output. We also know that the search for $u$ is successful iff
$q - b^{-t} < \beta_{j+2}$. Thus, if the search for $u$ is not successful, we have $\alpha < \beta_{j+2} \leq q - b^{-t} < q$,
and we can give the correct output $q - b^{-t}$.

\textit{Step 6B of the algorithm:} Exactly the same as 6A, 
but replace $j+1$ and $j+2$ by $j$ and $j+1$, respectively (\textit{end of Step 6B}).

The argument for correctness of Step 6B  is the same as for Step 6A,
just replace $j+1$ and $j+2$ by $j$ and $j+1$, respectively, and note that we have $\beta_{j+1}<q$ when the execution of the step starts.
\qed
\end{proof}

\begin{definition}
A function $D:\rational \rightarrow \{0,1\}$ is a {\em Dedekind
cut} of  the real number $\beta$ when $D(q)=0$ iff  $q< \beta$.
\end{definition}

\begin{corollary} \label{brandenburg}
Let $f$ be any  honest function.  
The  Dedekind
cut of  the real number $\alpha^f$ is elementary.
\end{corollary}

\begin{proof}
By Theorem \ref{evalouise} there is an elementary function $\check{T}$ such that
  $\check{T}(q) =0$ iff  $q< \alpha^f$. Let $D(q)=0$ if  $\check{T}(q) =0$, and let  $D(q)=1$ if  $\check{T}(q) \neq 0$.
The function $D$ is elementary since $\check{T}$ is elementary. Moreover, $D$ is the Dedekind cut of $\alpha^f$.
\qed
\end{proof}

\section{Main Results}

\begin{theorem} \label{rollerboller} For any subrecursive class $\xcs$ that is closed under primitive recursive operations, we have
$$(i)\;\; \xcs_{g\downarrow} \;\; \subset \;\; \bigcap_{b} \xcs_{b\downarrow}\;\;\;\mbox{ and } 
\;\;\;(ii)\;\;  \xcs_{g\uparrow} \;\; \subset \;\; \bigcap_{b} \xcs_{b\uparrow} .$$
\end{theorem}

\begin{proof}
The inclusion $\xcs_{g\uparrow} \subseteq \bigcap_{b} \xcs_{b\uparrow}$ is trivial.
Pick an honest function $f$ such that $f\not\in\xcs$.
By Theorem \ref{sebastian}, we have  $\alpha^f \in \bigcap_{b} \xcs_{b\uparrow}$.
By Theorem \ref{finbeck}, we have  $\alpha^f \not\in  \xcs_{g\uparrow}$.
This proves that $\xcs_{g\uparrow}  \subset  \bigcap_{b} \xcs_{b\uparrow}$.
The proof of (i) is symmetric.
\qed
\end{proof}

\begin{definition}
A function $\hat{T}:\rational \rightarrow \rational$  is  a {\em trace function from below} for  the irrational number $\alpha$ when 
we have $q <    \hat{T}(q) < \alpha$  for any $q<\alpha$.
A function $\check{T}:\rational \rightarrow \rational$  is  a {\em trace function from above} for  the irrational number $\alpha$ when 
we have $\alpha <    \check{T}(q) < q$  for any $q>\alpha$.
A function $T:\rational \rightarrow \rational$  is  a {\em trace function} for  the irrational number $\alpha$ when 
we have
$
\left\vert \alpha - q \right\vert >  \left\vert \alpha - T(q) \right\vert 
$
for any $q$.

For any  subrecursive class $\xcs$,
let $\xcs_D$ denote the set of irrational numbers that have a Dedekind cut in $\xcs$;
let $\xcs_{T\uparrow}$ denote the set of irrational numbers that have a trace function from below in $\xcs$;
let $\xcs_{T\downarrow}$ denote the set of irrational numbers that have a trace function from above in $\xcs$;
let $\xcs_T$ denote the set of irrational numbers that have a trace function in $\xcs$;
let $\xcs_{[\, ]}$ denote the set of irrational numbers that have a continued fraction in $\xcs$.
\end{definition}

It is proved in \cite{ciejour} that we have
$ \xcs_{g\downarrow} \,\cap \,  \xcs_{g\uparrow} = \xcs_T = \xcs_{[\,]}$
for any  $\xcs$ closed under primitive recursive operations.
It is conjectured in \cite{ciejour} that $ \xcs_{g\downarrow} \,\neq \, \xcs_{g\uparrow}$.
Theorem \ref{dilledalle} shows that this conjecture holds.
The next theorem will be used as a lemma in the proof of Theorem \ref{dilledalle}.

\begin{theorem} \label{mullemalle}
For any subrecursive class $\xcs$ that is closed under primitive recursive operations, we have 
$$(i)\;\; \xcs_{T\uparrow} \, \cap  \, \xcs_{D}  \;\; = \;\;  \xcs_{g\uparrow}        \;\;\;\mbox{ and } 
\;\;\;(ii)\;\;  \xcs_{T\downarrow} \, \cap  \, \xcs_{D}  \;\; = \;\;  \xcs_{g\downarrow}      \; .$$
\end{theorem}

\begin{proof}
We prove (i). The proof of (ii) is symmetric.  Let $\beta$ be an irrational number in the interval $(0,1)$,
 let $q$ be a rational number number in the same interval,
 and let $m,n\in\nat$ be such that $q=mn^{-1}$. We have $q < \beta$ iff $q\leq \hat{G}^\beta(n,1)$.
Let
$$
 \hat{T}(q) =  \hat{G}^\beta(n,1) \, + \, \hat{G}^\beta(n,2)
\;\;\; \mbox{ and } \;\;\; 
D(q) \; = \;  \begin{cases}
0 & \mbox{if  $q\leq \hat{G}^\beta(n,1)$}   \\
1 & \mbox{otherwise.}
\end{cases} 
$$
Then, $\hat{T}$ is a trace function from below for $\beta$, and $D$ is the Dedekind cut of $\beta$.
This proves the inclusion  $\xcs_{g\uparrow}   \subseteq \xcs_{T\uparrow}  \cap   \xcs_{D}$.
The inclusion holds for any $\xcs$ closed under elementary operations (and thus also for any $\xcs$ closed
under primitive recursive operations).

Let $D$ be the Dedekind cut of an irrational number $\beta$ in the interval $(0,1)$, and
let $\hat{T}$ be a trace function from below for $\beta$. We will now give an algorithm (that uses $D$ and  $\hat{T}$)
  for computing $\hat{G}^\beta(b,n+1)$
(the computation of $\hat{G}^\beta(b,0)$ is trivial). We leave it to the reader to check that the algorithm is correct.

\paragraph{Step 1 of the algorithm:} Compute the rational number $q$ such that 
$$q \;\; =  \;\; \hat{T}\left( \sum_{i=0}^n \hat{G}^\beta(b,i)  \right)$$
{\em (end of step 1)}.

The rational number $q$ can be written in the form
\begin{align}
q \;\; = \;\;  \sum_{i=0}^n \hat{G}^\beta(b,i)  \; + \; \xd b^{-k} \; + \; \epsilon \label{ebleskiver}
\end{align}
for some base-$b$ digit $\xd>0$, some $k\in\nat$ and some $\epsilon <  b^{-k}$. 

\paragraph{Step 2 of the algorithm:} Compute $k$ such that (\ref{ebleskiver}) holds. Use the Dedekind cut $D$ to search
for $\ell\leq k$ and $\xd\in\{1,\ldots, b-1\}$ such that
$$
\sum_{i=0}^n \hat{G}^\beta(b,i) \; + \; \xd b^{-\ell} \;\; <  \;\; \beta  \;\; <  \;\; \sum_{i=0}^n \hat{G}^\beta(b,i) \; + \; \xd b^{-\ell} 
\; + \;  b^{-\ell} \; .
$$
Let $\hat{G}^\beta(b,n+1) = \xd b^{-\ell}$
{\em (end of step 2)}.

The algorithm  above shows that we can compute $\hat{G}^\beta$ when $\hat{T}$ and $D$ are
available. The algorithm is not primitive recursive, but it is easy to see that the algorithm can be reduced to
a primitive recursive algorithm. (The algorithm uses $ \hat{G}^\beta(b,0),  \hat{G}^\beta(b,1),\ldots, \hat{G}^\beta(b,n)$ in the computation of
  $\hat{G}^\beta(b,n+1)$. A primitive recursive algorithm is required to only use  $\hat{G}^\beta(b,n)$ in the computation
of $\hat{G}^\beta(b,n+1)$.) Thus we conclude that
$\xcs_{T\uparrow}  \cap   \xcs_{D}  \subseteq \xcs_{g\uparrow}$ when $\xcs$ is closed under primitive recursive operations.
\qed
\end{proof}

\begin{theorem} \label{dilledalle}
For any   subrecursive class  $\xcs$ that is closed under primitive recursive operations,
there exist irrational numbers $\alpha$ and $\beta$ such that
$$(i)\;\; \alpha \in  \xcs_{g\downarrow} \setminus \xcs_{g\uparrow}       \;\;\;\mbox{ and } 
\;\;\;(ii)\;\;   \beta \in   \xcs_{g\uparrow} \setminus \xcs_{g\downarrow}        \; .$$
\end{theorem}

\begin{proof}
Pick an honest function $f$ such that $f\not\in\xcs$.
We have $\alpha^f \in \xcs_{T\downarrow}$ by Theorem \ref{evalouise}, and we have 
$\alpha^f \in\xcs_{D}$ by Corollary \ref{brandenburg}.
By Theorem \ref{mullemalle}, we have  $\alpha^f \in \xcs_{g\downarrow}$.
By Theorem \ref{finbeck}, we have  $\alpha^f \not\in  \xcs_{g\uparrow}$.
This proves (i). The proof of (ii) is symmetric.
\qed
\end{proof}

\end{document}